\documentclass[english, 12pt]{amsproc}
\usepackage{mathtext}

\usepackage[english]{babel}
\usepackage{amsthm,amsmath,amsfonts,amssymb,mathrsfs,amscd}
\usepackage{a4wide}

\usepackage{hyperref}
\usepackage{url}

\newcounter{stmcounter}[section]
\newcounter{thcounter}
\newcounter{problcounter}

\numberwithin{equation}{section}

\theoremstyle{plain}
\newtheorem{cor}[stmcounter]{Corollary}

\newtheorem{thm}[thcounter]{Theorem}

\newtheorem{prop}[stmcounter]{Proposition}

\newtheorem{probl}[problcounter]{Problem}

\theoremstyle{definition}

\theoremstyle{remark}
\newtheorem{ex}[stmcounter]{Example}
\newtheorem{rem}[stmcounter]{Remark}

\DeclareMathOperator{\Sp}{Sp}

\DeclareMathOperator{\Cone}{Cone}
\DeclareMathOperator{\conj}{conj}

\newcommand{\simc}{\!\!\sim}

\newcommand{\dd}{\partial}
\newcommand{\xb}{\mathbf{x}}
\newcommand{\tb}{\mathbf{t}}
\newcommand{\yb}{\mathbf{y}}
\newcommand{\hh}{\tilde{h}}

\newcommand{\Zo}{\mathbb{Z}}
\newcommand{\Ro}{\mathbb{R}}
\newcommand{\Rg}{\mathbb{R}_{\geqslant 0}}
\newcommand{\Co}{\mathbb{C}}
\newcommand{\Ho}{\mathbb{H}}

\newcommand{\Zt}{\Zo_2}
\newcommand{\Ko}{\mathbb{K}}

\newcommand{\CP}{\mathbb{C}P}
\newcommand{\HP}{\mathbb{H}P}
\newcommand{\OP}{\mathbb{O}P}

\begin{document}

\title[On actions of tori and quaternionic tori]{On actions of tori and quaternionic tori on products of spheres}
%\title[О действиях]{О действиях вещественного и кватернионного торов на произведениях сфер}

\author[A.\,A.\,Ayzenberg]{Anton Ayzenberg}
\address{Neapolis University, Pafos, Cyprus}
\email{ayzenberga@gmail.com}
%второй автор
\author[D.\,V.\,Gugnin]{Dmitry Gugnin}
\address{Steklov Mathematical Institute, Moscow, Russia}
\email{dmitry-gugnin@yandex.ru}

\date{\today}


\dedicatory{To Victor Buchstaber, \\ our great teacher and the wonderful person}


\thanks{This work was supported by the Russian Science Foundation under grant no. 23-11-00143 \href{https://rscf.ru/en/project/23-11-00143/}{https://rscf.ru/en/project/23-11-00143/}
}

\begin{abstract}
In this paper we study the actions of tori (standard compact tori, as well as their quaternionic analogues) on products of spheres. It is proved that the orbit space of a specific action of a torus on a product of spheres is homeomorphic to a sphere. A similar statement for a real torus $\mathbb{Z}_2^n$ was proved by the second author in 2019. We also provide a statement about arbitrary compact topological groups, generalizing the mentioned results, as well as the results of the first author about the actions of a compact torus of complexity one.
\end{abstract}

\maketitle

\section{Introduction}\label{secIntro}

Let $S^m = \{(x_1,\ldots, x_m,x_{m+1})\in \Ro^{m+1} \mid x_1^2+\ldots+ x_m^2 + x_{m+1}^2 = 1\}$ be a standard sphere. Let $\tau\colon S^m \to S^m$ denote one of standard involutions, namely $\tau(x_1,\ldots,x_m,x_{m+1}) = (x_1,\ldots,x_m, - x_{m+1})$. Consider the product of $k$ spheres of arbitrary positive dimensions:
\[
S^{m_1}\times S^{m_2}\times \ldots \times S^{m_k}, \quad k\geqslant 2.
\]
Commuting involutions $\tau_1, \ldots, \tau_k$ act on this manifold (the action is component-wise, each of the involutions is a permutation of the north and south poles of the corresponding sphere). Since involutions commute, we have a $C^{\omega}$-action of the group $\Zt^{k}$ on the product of spheres under consideration. The group $\Zt^{k}$ contains a subgroup $G_k$ of index 2, consisting of orientation-preserving elements.

The following result was obtained by the second author~\cite{Gug19}.

\begin{thm}[\cite{Gug19}]\label{thmReal}
The quotient space $S^{m_1}\times S^{m_2}\times \ldots \times S^{m_k}/G_k$ is homeomorphic to the sphere $S^m$, $m = m_1+\cdots + m_k$. In this case, the canonical projection onto the space of orbits is given by the following formula:
\begin{multline}
(x_{1,1}, \ldots, x_{m_1,1}, x_{m_{1} +1,1}, \ldots, x_{1,k}, \ldots, x_{m_k,k}, x_{m_k +1,k}) \mapsto\\ \mapsto
\frac{(x_{1,1}, \ldots, x_{m_1,1}, \ldots, x_{1,k}, \ldots, x_{m_k,k} ; x_{m_{1} +1,1} \cdots x_{m_k +1,k})}{\sqrt{ x_{1,1}^2 +  \ldots + x_{m_1,1}^2 + \ldots + x_{1,k}^2 + \ldots + x_{m_k,k}^2 + x_{m_{1} +1,1}^2 \cdots x_{m_k +1,k}^2 }}
\end{multline}
Moreover, the resulting branched covering $S^{m_1}\times S^{m_2}\times \ldots \times S^{m_k} \to S^m$ is globally $C^{\omega}$-smooth and nondegenerate at the points of the local homeomorphism.
\end{thm}

For the sake of convenience, we will call the group $\Zt^n=O(1)^n$ a real torus. Similarly, we have a complex\footnote{The adjective ``complex'' in this context refers to the fact that $T^1=U(1)$ is the set of complex numbers of length $1$. We do not assume any complex structures on complex tori. In this naming convention we are coherent with Arnold's mathematical ideology, where the ``real--complex--quaternionic'' trinity plays a crucial role.} torus $T^n=U(1)^n$, and a quaternionic torus $\Sp(1)^n$.

In this paper, we formulate and prove a natural generalization of the above theorem to the actions of complex and quaternionic tori; this is done in Section~\ref{secMainResults}. More generally, the result can be formulated for an arbitrary compact topological group; this is the subject of Section~\ref{secGeneralJoins}. From the general result we deduce that the quotient of some specific action of the quaternionic torus $\Sp(1)^{k-1}$ on the space $\Ho^k\cong \Ro^{4k}$ is homeomorphic to the space $\Ro^{k+3}$, see Proposition~\ref{propLocal}. This assertion is a quaternionic generalization of the previous result of the first author~\cite{AyzCompl} about toric actions of complexity one in general position.

\section{Main results}\label{secMainResults}

Consider $k\geqslant 2$ many spheres of dimensions at least $2$:
\begin{gather*}
S^{m_1} = \{ (\xb_1, z_1)\mid \xb_1 = (x_{1,1}, x_{2,1}, \ldots, x_{m_1-1,1}) \in \Ro^{m_1-1}, z_1\in \Co, |\xb_1|^2 + |z_1|^2 =1 \},\\
\vdots\\
S^{m_k} = \{ (\xb_k, z_k)\mid \xb_k = (x_{1,k}, x_{2,k}, \ldots, x_{m_k-1,k}) \in \Ro^{m_k-1}, z_k\in \Co, |\xb_k|^2 + |z_k|^2 =1 \}.
\end{gather*}
Consider the direct product $S^{m_1}\times S^{m_2}\times \ldots \times S^{m_k}$. It carries a (left) smooth action of the complex torus $T^{k-1}$. Namely, the element $(r_1,r_2,\ldots, r_{k-1})\in T^{k-1}$ translates a point
\[
((\xb_1, z_1), (\xb_2, z_2), \ldots, (\xb_k, z_k)) \in S^{m_1}\times S^{m_2}\times \ldots \times S^{m_k}
\]
to the point
\[
((\xb_1, z_1r_1^{-1}), (\xb_2, r_1z_2r_2^{-1}), (\xb_3, r_2z_3r_3^{-1}), \ldots, (\xb_k, r_{k-1}z_k)).
\]

\begin{thm}\label{thmComplex}
The quotient space $S^{m_1}\times S^{m_2}\times \ldots \times S^{m_k}/T^{k-1}$ is homeomorphic to the sphere $S^m, m = m_1+\ldots + m_k - (k-1)$. The canonical projection to the orbit space is given by the formula:
\begin{equation}\label{eqStar}
((\xb_1, z_1), (\xb_2, z_2), \ldots, (\xb_k, z_k)) \mapsto \frac{(\xb_1, \xb_2, \ldots, \xb_k, z_1z_2\ldots z_k)}{\sqrt{ |\xb_1|^2 + |\xb_2|^2 +  \ldots +  |\xb_k|^2 + |z_1z_2\ldots z_k|^2 }}
\end{equation}
\end{thm}

The proof of this theorem is completely analogous to the proof of its quaternionic version, Theorem~\ref{thmQuaternionic} below. Let $\Ho$ denote the algebra of quaternions and $\Sp(1)$ --- the Lie group of unit quaternions (i.e. the quaternions of unit length). Consider $k\geqslant 2$ many spheres of dimensions at least 4:
\begin{gather*}
S^{m_1} = \{ (\xb_1, q_1)\mid \xb_1 = (x_{1,1}, x_{2,1}, \ldots, x_{m_1-3,1}) \in \Ro^{m_1-3}, q_1\in \Ho, |\xb_1|^2 + |q_1|^2 =1 \}, \\
\vdots \\
S^{m_k} = \{ (\xb_k, q_k)\mid \xb_k = (x_{1,k}, x_{2,k}, \ldots, x_{m_k-3,k}) \in \Ro^{m_k-3}, q_k\in \Ho, |\xb_k|^2 + |q_k|^2 =1 \}.
\end{gather*}

Consider the direct product $S^{m_1}\times S^{m_2}\times \ldots \times S^{m_k}$. It carries a (left) smooth action of the quaternionic torus of rank $k-1$
\[
\Sp(1)^{k-1} = \underbrace{\Sp(1)\times \Sp(1)\times \ldots \times \Sp(1)}_{k-1 \text{ times}}.
\]
Namely, the element $(r_1,r_2,\ldots, r_{k-1})\in \Sp(1)^{k-1}$ translates a point
\[
((\xb_1, q_1), (\xb_2, q_2), \ldots, (\xb_k, q_k)) \in S^{m_1}\times S^{m_2}\times \ldots \times S^{m_k}
\]
to the point
\[
((\xb_1, q_1r_1^{-1}), (\xb_2, r_1q_2r_2^{-1}), (\xb_3, r_2q_3r_3^{-1}), \ldots, (\xb_k, r_{k-1}q_k)).
\]

\begin{thm}\label{thmQuaternionic}
The quotient space $S^{m_1}\times S^{m_2}\times \ldots \times S^{m_k}/\Sp(1)^{k-1}$ is homeomorphic to the sphere $S^m, m = m_1+\ldots + m_k - 3(k-1)$. The canonical projection to the orbit space is given by the formula:
\begin{equation}\label{eqTwoStars}
((\xb_1, q_1), (\xb_2, q_2), \ldots, (\xb_k, q_k)) \mapsto \frac{(\xb_1, \xb_2, \ldots, \xb_k, q_1q_2\ldots q_k)}{\sqrt{ |\xb_1|^2 + |\xb_2|^2 +  \ldots +  |\xb_k|^2 + |q_1q_2\ldots q_k|^2 }}.
\end{equation}
\end{thm}

\begin{proof}
First, let us prove that the canonical projection onto the space of orbits is realized by a simpler formula:
\begin{equation}\label{eq1}
((\xb_1, q_1), (\xb_2, q_2), \ldots, (\xb_k, q_k)) \mapsto (\xb_1, \xb_2, \ldots, \xb_k, q_1q_2\ldots q_k).
\end{equation}
From the definition of the action of the quaternionic torus on the product of spheres, it is clear that any orbit gets mapped to a single point under the map~\eqref{eq1}. Let us prove that two different orbits cannot map to the same point under~\eqref{eq1}. Indeed, assume
\begin{equation}\label{eq2}
(\xb_1, \xb_2, \ldots, \xb_k, q_1q_2\ldots q_k) = (\yb_1, \yb_2, \ldots, \yb_k, p_1p_2\ldots p_k).
\end{equation}
It follows that $\xb_i = \yb_i, 1\leqslant i\leqslant k$. Therefore, $|q_i|=|p_i|$, $1\leqslant i\leqslant k$. It is easily seen that, remaining within a single orbit, we can assume that $0\leqslant q_i=p_i\leqslant 1$, $1\leqslant i\leqslant k-1$.

At first, let us assume that $q_1q_2\ldots q_k = p_1p_2\ldots p_k \neq 0$. Then, obviously, $q_k=p_k$, which was to be proved. Now consider the case $q_1q_2\ldots q_k = p_1p_2\ldots p_k = 0$. If $q_k = p_k = 0$, then everything is proved. Let $|q_k| = |p_k| > 0$. Denote the unit quaternion $p_kq_k^{-1}$ by $a$.

Let $q_{k-1}=p_{k-1}=0$. Then one can translate the tuple $(q_1,q_2, \ldots, 0, q_k)$ to the tuple $(p_1,p_2,\ldots, 0, p_k)$ using the element $(1,1,\ldots, 1, r_{k-1}= a)\in \Sp(1)^{k-1}$. If $q_{k-1} = p_{k-1} > 0$, $q_{k-2} = p_{k-2} = 0$, then $(1,1,\ldots,1,a,a)$ is the required element of the group $\Sp(1)^{k-1}$. If $q_{k-1} = p_{k-1} > 0$, $q_{k-2} = p_{k-2} > 0$, $q_{k-3} = p_{k-3} = 0$, then the required element is $(1,1,\ldots,1,a,a,a)$. Similar arguments work for other cases. In the most extreme case, we have $q_{k-1} = p_{k-1} > 0$, $q_{k-2} = p_{k-2} > 0$, $\ldots$, $q_2 = p_2 > 0$, $q_1 = p_1 = 0$. In this case, the required element is $(a,a,\ldots,a)$. Thus, we have proved that formula~\eqref{eq1} determines a well-defined canonical projection onto the orbit space.

It can be seen that the vector on the right hand side of formula~\eqref{eq1} is nonzero. Since the formula~\eqref{eqTwoStars} on the left hand side involves a compact Hausdorff space (the quotient of a Hausdorff space by a continuous action of a compact Lie group is Hausdorff~\cite{Bredon-ru}), and the sphere appears on the right hand side, it suffices to check {\it bijectivity} of the map~\eqref{eqTwoStars}.

\textbf{Injectivity}. We need to check that the equality
\begin{equation}\label{eq3}
(\xb_1, \xb_2, \ldots, \xb_k, q_1q_2\ldots q_k) = \mu (\yb_1, \yb_2, \ldots, \yb_k, p_1p_2\ldots p_k),\quad \mu>1
\end{equation}
never occurs for distinct points. We have $q_1q_2\ldots q_k = \mu p_1p_2\ldots p_k$. Two cases are possible: (A) $q_1q_2\ldots q_k = p_1p_2\ldots p_k = 0$, and (B) $q_i\neq 0, p_i\neq 0, 1\leqslant i \leqslant k$.

\textit{Case (A).} We have $p_{i_0} = 0$ for some $1\leqslant i_0\leqslant k$. Then $|\yb_{i_0}| = 1$ and $|\xb_{i_0}| = \mu |\yb_{i_0}| = \mu > 1$, which is impossible.
\textit{Case (B).} We have $|\xb_i| = \mu |\yb_i| \geqslant |\yb_i|$, $1\leqslant i\leqslant k$. This implies that $0 < |q_i|\leqslant |p_i|$, $1\leqslant i\leqslant k$. Hence $0 < |q_1q_2\ldots q_k| \leqslant |p_1p_2\ldots p_k|$. On the other hand, we have $|q_1q_2\ldots q_k| = \mu |p_1p_2\ldots p_k| > |p_1p_2\ldots p_k|$, --- a contradiction.

\textbf{Surjectivity}. Since the image of a compact space under a continuous mapping is always a compact space, then either the surjectivity is proved or the image of the map~\eqref{eqTwoStars} is a proper subcompact in the sphere $S^m$. Again, we argue from the contrary. Assume there exists a vector $(\tb_1, \tb_2, \ldots, \tb_k, t)$ such that $\tb_i \neq 0$, $1\leqslant i\leqslant k$, $t\neq 0$, and for any $\mu > 0$ the vector $\mu(\tb_1, \tb_2, \ldots, \tb_k, t)$ does not have the form $(\xb_1, \xb_2, \ldots, \xb_k, q_1q_2\ldots q_k)$. Denote $\min_{1\leqslant i\leqslant k}\{1 /|\tb_i|\}$ by $\mu_0$.

As the parameter $\mu$ runs over the interval $(0,\mu_0)$, the lengths of the vectors $\mu \tb_i = \xb_i$, $1\leqslant i\leqslant k$ increase strictly and continuously and run over the intervals $(0, \mu_0 |\tb_i|) \subset (0,1)$, $1\leqslant i\leqslant k$. Moreover, there exists $1\leqslant i_0\leqslant k$ such that $(0, \mu_0| \tb_{i_0} |) = (0,1)$. It follows from the length expressions $|q_i| = \sqrt{1- |\xb_i|^2}$, that the length of $|q_1(\mu)q_2(\mu)\ldots q_k(\mu)|$ decreases strictly and continuously from $1$ to $0$ (not taking extreme values). In this case, it is possible to achieve collinearity of the nonzero quaternions $t$ and $q_1(\mu)q_2(\mu)\ldots q_k(\mu)$, $0< \mu < \mu_0$. Since the length of $|\mu t|$ increases strictly and continuously from $0$ to $|\mu_0||t|>0$, the Cauchy intermediate value theorem asserts that there exists a parameter $\mu_1\in (0, \mu_0)$, for which the equality $\mu_1(\tb_1, \tb_2, \ldots, \tb_k, t) = (\xb_1, \xb_2, \ldots, \xb_k, q_1(\mu_1)q_2(\mu_1)\ldots q_k(\mu_1))$ holds.

The theorem is completely proven.
\end{proof}

In the work~\cite{Gug19} of the second author, the proof of Theorem~\ref{thmReal} was omitted due to its simplicity. However, we still give here the proof of the most nontrivial part, namely, the non-degeneracy (local diffeomorphism) of the corresponding branched covering at the points of the local homeomorphism (away from the branching locus).

\begin{proof}
It is easily shown (similar to the above reasoning) that the canonical projection onto the orbit space is given by a simpler formula:
\begin{multline*}
(x_{1,1}, \ldots, x_{m_1,1}, q_1, x_{1,2},\ldots, x_{m_2,2}, q_2, \ldots, x_{1,k},\ldots, x_{m_k,k}, q_k) \mapsto \\
  \mapsto (x_{1,1}, \ldots, x_{m_1,1}, x_{1,2},\ldots, x_{m_2,2}, \ldots, x_{1,k},\ldots, x_{m_k,k}, q_1q_2\cdots q_k).
\end{multline*}
Here, we denote $x_{m_i+1,i} = q_i\in \Ro$, $1\leqslant i\leqslant k$ for the sake of simplicity. Moreover, for the initial map onto the unit sphere, the points of the local homeomorphism are either $(A)$ points with $q_1q_2\cdots q_k \neq 0$, or $(B)$ there is a unique $1\leqslant j \leqslant k$ with $q_j=0$.

It is understood that the points of a local diffeomorphism for the original map of smooth $m$-dimensional manifolds $S^{m_1}\times \ldots \times S^{m_k} \to S^m$ correspond (in both directions) to the points of a local diffeomorphism for the following map of smooth $(m+1)$-dimensional manifolds $F\colon S^{m_1}\times \ldots \times S^{m_k}\times (0,+\infty) \to \Ro^m\setminus\{0\}$:
\begin{gather*}
F(x_{1,1}, \ldots, x_{m_1,1}, q_1, x_{1,2},\ldots, x_{m_2,2}, q_2, \ldots, x_{1,k},\ldots, x_{m_k,k}, q_k;\mu) =\\
\mu(x_{1,1}, \ldots, x_{m_1,1}, x_{1,2},\ldots, x_{m_2,2}, \ldots, x_{1,k},\ldots, x_{m_k,k}, q_1q_2\cdots q_k),
\end{gather*}
where the parameter $\mu$ can be taken arbitrarily. Let us verify that in both cases $(A)$ and $(B)$ the Jacobian of the map $F$ is nonzero.

\textbf{Case $(A)$.} In this case, the string $(x_{1,1}, \ldots, x_{m_1,1}, x_{1,2},\ldots, x_{m_2,2}, \ldots, x_{1,k},\ldots, x_{m_k,k};\mu)$ can be taken as the local coordinates in the preimage. We need to calculate the determinant of order $m+1$.

The last column of the desired determinant ($\dd F(\ldots)/\dd \mu$) is equal to (the transposed string)
\[
(x_{1,1}, \ldots, x_{m_1,1}, \ldots, x_{1,k},\ldots, x_{m_k,k}, q_1q_2\cdots q_k)^\intercal.
\]
Further, the column $\dd F(\ldots)/\dd x_{1,1}$ divided by $\mu$ equals
\[
(1,0,\ldots,0, (\dd q_1/\dd x_{1,1}) q_2q_3\cdots q_k)^\intercal = \left(1,0,\ldots,0,  -\frac{x_{1,1}}{q_1} q_2q_3\cdots q_k \right)^\intercal.
\]
Similarly, the column $\dd F(\ldots)/\dd x_{2,1}$ divided by $\mu$ equals
\[
(0,1,0, \ldots,0, (\dd q_1/\dd x_{2,1}) q_2q_3\cdots q_k)^\intercal = \left(0,1,0,\ldots,0,  -\frac{x_{2,1}}{q_1} q_2q_3\cdots q_k \right)^\intercal.
\]
Making further calculations, we get that the penultimate column $\dd F(\ldots)/\dd x_{m_k,k}$ divided by $\mu$ equals
\[
(0,0, \ldots,0,1, (\dd q_k/\dd x_{m_k,k}) q_1q_2\cdots q_{k-1})^\intercal = \left(0,0,\ldots,0,1,-\frac{x_{m_k,k}}{q_k} q_1q_2\cdots q_{k-1}\right)^\intercal.
\]
Subtracting from the last column $(x_{1,1}, \ldots, x_{m_1,1}, \ldots, x_{1,k},\ldots, x_{m_k,k}, q_1q_2\cdots q_k)$ the first column multiplied by $x_{1,1}$, the second column multiplied by $x_{2,1}$, etc, we obtain, as a result, the lower triangular matrix with the diagonal
\[
\left(1,1,\ldots, 1, q + \frac{x_{1,1}^2 + x_{2,1}^2 + \ldots + x_{m_1,1}^2}{q_1^2}q +  \ldots + \frac{x_{1,k}^2 + x_{2,k}^2 + \ldots + x_{m_k,k}^2}{q_k^2}q\right),
\]
where $q=q_1q_2\cdots q_k$. Since $q\neq 0$, the determinant of this matrix is nonzero. The required local diffeomorphism in the case $(A)$ is proved.

\textbf{Case $(B)$.} Due to certain symmetry of the function $F$ in its arguments, we can assume without loss of generality that $q_1=0$, $q_2q_3\cdots q_k \neq 0$, $x_{1,1} \neq 0$. In this situation, one can choose the string
\[
q_1, x_{2,1}, x_{3,1}, \ldots, x_{m_1,1}, x_{1,2}, \ldots, x_{m_2,2}, \ldots, x_{1,k}, \ldots, x_{m_k,k}, \mu
\]
as local coordinates in the preimage. Recall the definition of $F$:
\[
F(\ldots) = \mu(x_{i,j}; q).
\]
Let us write down the calculations of all first partial derivatives of the function $F$:
\begin{gather*}
  \frac{1}{\mu}\frac{\dd F}{\dd q_1} = (0, 0, \ldots, 0; q_2q_3\cdots q_k), \\
  \frac{\dd F}{\dd \mu} = (x_{i,j};0), \\
  \frac{1}{\mu}\frac{\dd F}{\dd {x_{2,1}}} = \left(-\frac{x_{2,1}}{x_{1,1}}, 1, 0,\ldots, 0; 0\right), \  \frac{1}{\mu}\frac{\dd F}{\dd {x_{3,1}}} = \left(-\frac{x_{3,1}}{x_{1,1}}, 0, 1,0, \ldots, 0; 0\right),  \ldots , \\
  \frac{1}{\mu}\frac{\dd F}{\dd {x_{m_1,1}}} = \left(-\frac{x_{m_1,1}}{x_{1,1}}, 0, 0,\ldots, 0, 1, 0, \ldots, 0; 0\right), \\
  \frac{1}{\mu}\frac{\dd F}{\dd x_{1,2}} = (0, \ldots, 0,1,0, \ldots, 0; 0), \ \ldots \ ,   \frac{1}{\mu}\frac{\dd F}{\dd x_{m_k,k}} = (0, \ldots,0, 1;0).
\end{gather*}

By carefully computing the Jacobian of the map $F$ at the given point, we can see that it is nonzero if and only if the following determinant of order $m_1$ is nonzero:
\[
\begin{vmatrix}
x_{1,1}& -x_{2,1}  & -x_{3,1} & \ldots & -x_{m_1-1,1} & -x_{m_1,1}\\
x_{2,1}& x_{1,1} & 0 & \ldots & 0 & 0\\
x_{3,1}& 0 & x_{1,1} &  \ldots & 0 & 0 \\
\vdots& \vdots & \vdots &\ddots  & \vdots  & \vdots\\
x_{m_1-1,1} & 0  & 0 & \ldots & x_{1,1}  & 0\\
x_{m_1,1}& 0 & 0 &\ldots & 0 &  x_{1,1}
\end{vmatrix}
\]
Denote by $A$ the matrix under this determinant.

If this determinant is zero, then the skew-Hermitian matrix $A-x_{1,1}E$ has a nonzero real eigenvalue $\lambda = -x_{1,1}$. However, it is well known, that all eigenvalues of a skew-Hermitian matrix are purely imaginary complex numbers, and it was assumed earlier that $x_1\neq 0$. This contradiction shows that the Jacobian under consideration is nonzero. The desired local diffeomorphism in the case of $(B)$ is completely proved.
\end{proof}

\begin{rem}\label{remSubmersion}
Similar to the proof above one can show that the maps~\eqref{eqStar} and~\eqref{eqTwoStars} to the orbit space from Theorem~\ref{thmComplex} (complex tori) and Theorem~\ref{thmQuaternionic} (quaternionic tori) are smooth and they are submersions outside the degeneration locus. This means that in the open set of free orbits, the differentials of these maps have maximal possible rank equal to the dimension of the orbit space.
\end{rem}

In the work~\cite{Gug23b} of the second author, it was shown that the number $k-1$ of commuting involutions on the product of spheres $S^{m_1}\times \ldots \times S^{m_k}$ is the minimal possible if one wants to obtain a rational homological sphere as the orbit space. We pose the following problem related to this fact.

\begin{probl}
Is it true that, for an arbitrary smooth action of the complex torus $T^{k-2}$ on the product $S^{m_1}\times \ldots \times S^{m_k}$ of spheres of dimensions $\geqslant2$, the corresponding orbit space is not a (rational homological) sphere? Is it true that for an arbitrary smooth action of the quaternionic torus $\Sp(1)^{k-2}$ on the product $S^{m_1}\times \cdots \times S^{m_k}$ of spheres of dimensions $\geqslant 4$ the corresponding orbit space is not a (rational homology) sphere?
\end{probl}

This conjecture seems hard. It is nontrivial even in the simplest case $k=3$.

\section{General groups}\label{secGeneralJoins}

Consider nonempty compact Hausdorff spaces $X_1,X_2,\ldots, X_k$, $k\geqslant 2$, and an arbitrary compact Hausdorff group $G$. Consider the following (left) continuous action of the group $G^{k-1}$ on the product of joins $\prod_{i=1}^{k}(X_i\ast G)$.
Namely, the element $(r_1,r_2,\ldots, r_{k-1})\in G^{k-1}$ translates a point
\[
(x_1,1-t_1, q_1,t_1; x_2,1-t_2, q_2,t_2; \ldots; x_k,1-t_k, q_k,t_k)
\]
to the point
\begin{equation}\label{eqCodiagonalAcn}
(x_1,1-t_1, q_1r_1^{-1},t_1; x_2,1-t_2, r_1q_2r_2^{-1},t_2; \ldots; x_k,1-t_k, r_{k-1}q_k,t_k).
\end{equation}

\begin{thm}\label{thmJoins}
The orbit space $\prod_{i=1}^{k}(X_i\ast G)/G^{k-1}$ is homeomorphic to the join
$X_1\ast X_2\ast \ldots\ast X_k\ast G$. The canonical projection onto the space of orbits is given by the formula:
\begin{multline}\label{eqJoinMap}
(x_1,1-t_1, q_1,t_1; x_2,1-t_2, q_2,t_2; \ldots; x_k,1-t_k, q_k,t_k) \mapsto \\
\mapsto (x_1,1-t_1;x_2,1-t_2;\ldots; x_k,1-t_k; q_1q_2\cdots q_k, t_1t_2\cdots t_k)/A,
\end{multline}
where $A=t_1\cdots t_k+\sum\nolimits_{i=1}^{k}(1-t_i)$ is the normalizing factor.
\end{thm}

\begin{proof}
Recall that the join of spaces $Y_1,\ldots,Y_m$ is the identification space
\[
Y_1\times\cdots\times Y_s\times \Delta^{m-1}/\simc,
\]
where $\Delta^{m-1}$ is the standard simplex with barycentric coordinates $(s_1,\ldots,s_m)$, $s_i\geqslant 0$, $\sum s_i=1$, and the equivalence relation $\sim$ is generated by the conditions
\[
(y_1,\ldots,y_l,\ldots,y_m,(s_1,\ldots,s_m))\sim (y_1,\ldots,y_l',\ldots,y_m,(s_1,\ldots,s_m)), \text{ if }s_l=0
\]
for any $l\in[m]$. Notice that $t_i\in[0;1]$ in~\eqref{eqJoinMap}, therefore $A>0$. Therefore, all coefficients $\frac{1-t_i}{A}$, $i\in[k]$ and $\frac{t_1\cdots t_k}{A}$ are nonnegative and sum to $1$ due to the choice of the normalizing factor $A$. Hence formula~\eqref{eqJoinMap} provides a well-defined continuous map of the form
\[
h\colon\prod\nolimits_{i=1}^{k}(X_i\times \Delta^1\times G)\to X_1\times\cdots\times X_k\times G\times \Delta^{k-1}.
\]
It is easily seen that the map $h$ is surjective. Let us check that $h$ descends to a well-defined map of the quotient spaces, the joins from the statement of the theorem. If $t_i=1$ for some $i\in[k]$, then the corresponding factor $X_i$ collapses to one point both in the space $X_i\ast G$ and in the space $X_i\ast\cdots \ast X_k\ast G$, so equivalent points of this kind are mapped to equivalent points. If $t_i=0$, then $t_1\cdots t_k=0$, and, similarly, equivalent points are mapped to equivalent ones. Therefore, the formula~\eqref{eqJoinMap} descends to a well-defined continuous map
\[
\hh\colon (X_1\ast G)\times (X_2\ast G)\times \ldots \times (X_k\ast G)/G^{k-1} \to X_1\ast X_2\ast \ldots\ast X_k\ast G.
\]
Since all the spaces appearing in this formula are Hausdorff compact, it suffices to show that the map $\hh$ is bijective. The surjectivity of $\hh$ follows from the surjectivity of~$h$.

Let us prove that $\hh$ is injective. Generally, the proof is similar to the proof of Theorem~\ref{thmQuaternionic}. Assume that
$\hh(x_1,t_1,q_1,\ldots,x_k,t_k,q_k) = \hh(x_1',t_1', q_1',\ldots, x_k',t_k',q_k')$. Looking at the real parameters $t_1,\ldots,t_k$ we see that formula~\eqref{eqJoinMap} defines a homeomorphism of the $k$-dimensional cube onto the $k$-dimensional simplex. Therefore the equalities $t_i=t_i'$ hold for all $i\in[k]$.

If $0<t_i<1$ for all $i$, then the assertion of injectivity reduces to the homeomorphism $G^k/G^{k-1}\cong G$ given by the formula $[(q_1,\ldots,q_k)]\mapsto q_1\cdots q_k$. If $t_i=1$ for some $i$, then the fiber $X_i$ collapses on both sides of the map. If $t_1\cdots t_k=0$, then $t_i=0$ for some $i\in[k]$. This means that the $i$-th component $G_i$ of the product $G^k$ collapses in the target of $\hh$. Taking the quotient of $G^k$ simultaneously by $G_i$ and the action of the subgroup $G^{k-1}$ described by the formula~\eqref{eqCodiagonalAcn}, we see that the entire group $G^k$ collapses into a point. Therefore, the identifications on the face of the simplex $\{t_1\cdots t_k=0\}$ are the same in the space $X_1\ast X_2\ast \ldots\ast X_k\ast G$ and in the space $(X_1\ast G)\times (X_2\ast G)\times \ldots \times (X_k\ast G)/G^{k-1}$. Injectivity is proved.
\end{proof}

Note that the extreme trivial cases of Theorem~\ref{thmJoins} appear to be informative.

\begin{ex}\label{exTrivialGroup}
Let us apply Theorem~\ref{thmJoins} to the trivial group $G=\{1\}$. We get the standard topological fact:
\[
\Cone X_1\times \Cone X_2\times \cdots \times \Cone X_k\cong \Cone(X_1\ast X_2\ast \ldots\ast X_k).
\]
\end{ex}

\begin{ex}\label{exTrivialSpace}
Let us apply Theorem~\ref{thmJoins} to the one-point topological spaces $X_i=\ast$, $i\in[k]$. We get a homeomorphism:
\[
(\Cone G)^{k}/G^{k-1}\cong\underbrace{\Cone\cdots\Cone}_k G\cong \Delta^{k-1}\ast G.
\]
\end{ex}

\begin{rem}\label{remParticCases}
The topological part of Theorems~\ref{thmReal}, \ref{thmComplex} and \ref{thmQuaternionic} is a special case of Theorem~\ref{thmJoins} if we let $G$ be a torus: real $\Zt=O(1)$, complex $T^1\cong U(1)$, or quaternionic $\Sp(1)$, respectively; and $X_i$ --- spheres of arbitrary dimensions.

Note, however, that the above theorems contain stronger assertions. Since both the spaces $X_i$ and the group $G$ are spheres, their joins are also spheres, and therefore have a natural smooth structure. The question of whether the natural projection onto the orbit space is smooth seems important. For this reason we spent some time proving smoothness in Section~\ref{secMainResults}.
\end{rem}

Example~\ref{exTrivialSpace} has a useful treatment in the real, complex, and quaternionic cases. Let $\Ko$ denote $\Ro$, $\Co$, or $\Ho$, the number $d(\Ko)$ be equal to $1$, $2$, or $4$, respectively, and $S(\Ko)$ denote the compact Lie group of numbers having norm $1$ in the corresponding division algebra. Thus $S(\Ko)$ is $\Zt$, $U(1)$, or $\Sp(1)$. Topologically, $S(\Ko)$ is a sphere of dimension $d(\Ko)-1$.

The group $S(\Ko)^{k-1}$ acts linearly on $\Ko^k\cong \Ro^{d(\Ko)k}$ by the formula
\begin{equation}\label{eqCodiagOnEuclid}
(g_1,\ldots,g_{k-1})(q_1,\ldots,q_k)=(q_1g_1^{-1},g_1q_2g_2^{-1},\ldots,g_{k-2}q_{k-1}g_{k-1}^{-1},g_{k}q_{k}).
\end{equation}

\begin{prop}\label{propLocal}
The orbit space $\Ko^k/S(\Ko)^{k-1}$ is homeomorphic to the space $\Ro^{d(\Ko)+k-1}$.
\end{prop}

\begin{proof}
\textbf{First version of the proof.} The space $\Ko$ is equi\-va\-ri\-ant\-ly diffeomorphic to the open unit ball in $\Ko$ that is the interior of the cone $\Cone S(\Ko)$. Apply the homeomorphism from Example~\ref{exTrivialSpace} and pass to the interiors of the spaces.

\textbf{Second version of the proof.} Consider the coordinate-wise action of $S(\Ko)^k$ on $\Ko^k$. We have $\Ko^k/S(\Ko)^k\cong \Rg^k$. The projection to the quotient space has a natural section, so the space $\Ko^k$ can be represented as the identification space
\begin{equation}\label{eqQuotConstr}
\Ko^k\cong \Rg^k\times S(\Ko)^k/\simc,
\end{equation}
similarly to how quasitoric manifolds are defined in the toric topology~\cite{BPnew}. Moding out the second factor in the construction~\eqref{eqQuotConstr} by the torus action~\eqref{eqCodiagOnEuclid}, we obtain
\begin{equation}\label{eqQuotReduced}
\Ko^k/S(\Ko)^{k-1}\cong \Rg^k\times S(\Ko)/\simc.
\end{equation}
Note that $\Rg^k$ is homeomorphic to the half-space $\Ro^{k-1}\times \Rg$. A careful analysis of the stabilizers shows that the relation $\sim$ in the formula~\eqref{eqQuotReduced} collapses the component $S(\Ko)$ into a point if and only if the corresponding point from $\Ro^{k-1}\times \Rg$ belongs to the boundary $\Ro^{k-1}\times\{0\}$, see~\cite{AyzCompl} for details. Hence we get
\[
\Ko^k/S(\Ko)^{k-1}\cong \Ro^{k-1}\times (\Rg\times S(\Ko)/\simc)\cong \Ro^{k-1}\times\Ko
\]
which completes the proof.

%\textbf{Third version of the proof.} Consider a noncompact version of the group $S(\Ko)^{k-1}$, namely the group $(\Ko^\times)^{k-1}$, where $\Ko^\times$ is the multiplicative group of nonzero elements of the division algebra $\Ko$. The action of $(\Ko^\times)^{k-1}$ on $\Ko^k$ is given by the same formula~\eqref{eqCodiagOnEuclid}. Notice that the quotient space $\Ko^k/(\Ko^\times)^{k-1}$ is naturally homeomorphic to the space $\Ko\cong \Ro^{d(\Ko)}$, since the quotient map simply coincides with the map $\Ko\times\cdots\times\Ko\to \Ko$ given by the formula $(q_1,\ldots,q_k)\mapsto q_1\cdots q_k$. The quotient $\Ko^k/(\Ko^\times)^{k-1}$ coincides with the quotient of the desired space $\Ko^k/S(\Ko)^{k-1}$ by the residual free action of a noncompact group $\Ro_{>}^{k-1}$, which is isomorphic, as a Lie group, to the additive group $\Ro^{k-1}$. Since the latter action is free and the base is homeomorphic to $\Ro^{d(\Ko)}$, we conclude that the total space $\Ko^k/S(\Ko)^{k-1}$ is homeomorphic to the space $ \Ro^{d(\Ko)}\times \Ro^{k-1}$.
\end{proof}

\begin{rem}\label{remComplexAndReal}
The second version of the proof in the general case is completely analogous to the complex case considered in~\cite[Lem.2.11]{AyzCompl}, see also~\cite[Thm.3.6]{Styrt}. The real case was studied in detail by Mikhailova~\cite[Thm.2.2]{Mikh}.
%См. также работу~\cite{Gorch}, посвященную глобальной топологии пространств орбит действия дискретного тора.
\end{rem}

%\begin{rem}\label{remMomAng}
%The third version of the proof is conceptually similar to the method of defining geometric structures on (compact) moment-angle manifolds via actions of noncompact groups on noncompact spaces, see~\cite{PanGeom}.
%\end{rem}

Proposition~\ref{propLocal} can be understood as a local result. The global consequence follows.

\begin{cor}\label{corOrbitGlobal}
Consider a smooth action of a torus $G$ (real, complex, or quaternionic) on a closed smooth manifold $X$. Assume that the linearized action on the normal space to each orbit is equivalent to a representation of the form~\eqref{eqCodiagOnEuclid} multiplied by a trivial representation. Then the orbit space $X/G$ is a topological manifold.
\end{cor}

In the case of a complex torus, many examples of such actions with isolated fixed points were studied in the works of the first author~\cite{AyzCompl,AyzHP,AyzMasEquiv}. Actions of real tori whose orbit spaces are manifolds were studied by Gorchakov~\cite{Gorch}.

There is also a series of results worth mentioning in this context:
\begin{equation}\label{eqAtiyahArnold}
\CP^2/\conj\cong S^4,\qquad \HP^2/U(1)\cong S^7,\qquad \OP^2/\Sp(1)\cong S^{13}.
\end{equation}
Here the first homeomorphism is the classical Kuiper--Massey theorem (Arnold~\cite{Arn} attributes this result to Pontryagin), the second homeomorphism $\HP^2/U(1)\cong S^7 $ is the result of Arnold himself~\cite[Ex.4]{Arn}, and the third one is due to Atiyah--Berndt~\cite{AtBer}. In these examples, the set of fixed points is not discrete, but the linearization of the group action on a normal space to the fixed points' submanifold is equivalent to the linear representations of $\Zt$ on $\Ro^2$, $U(1)$ on $\Co^2=\Ro^4$, and $\Sp(1)$ on $\Ho^2 =\Ro^8$ respectively.

Corollary~\ref{corOrbitGlobal} explains why the orbit spaces in all cases~\eqref{eqAtiyahArnold} are manifolds; although, by no means, it explains why the orbit spaces are homeomorphic to spheres. In view of Theorems~\ref{thmReal}, \ref{thmComplex}, \ref{thmQuaternionic} and homeomorphisms~\ref{eqAtiyahArnold}, arises a natural question.

\begin{probl}
Describe a class of actions of the groups $\Zt^k$, $T^k$, and $\Sp(1)^k$ on smooth manifolds with orbit spaces homeomorphic to spheres which is general enough to include both the products of spheres and the manifolds $\CP^2$, $\HP^2$, and $\OP^2$.
\end{probl}

% Если В.М. действительно внес вклад в эту работу, можешь раскомментить благодарность. Я просто категорически против благодарностей "за поддержку и постоянное внимание", которые пишутся по инерции. Потому что они обесценивают те случаи, когда людей действительно есть за что благодарить - например, если они подсказали какую-то важную идею, но не включены в соавторы. Посвящение, конечно, ок. Думаю, этого достаточно.
%
%\section*{Благодарности} Авторы выражают глубокую признательность чл.-корр. РАН, профессору В.\,М.\,Бух\-шта\-бе\-ру за поддержку и постоянное внимание в процессе написания работы.

\end{document}